\documentclass[]{amsart}
 \usepackage[]{amsmath, amsthm, amsfonts,amssymb}
 \usepackage[all]{xy}
 \usepackage{hyperref}
 \usepackage{srcltx}
 \xyoption{curve}
 \usepackage[mathscr]{eucal}

 \newcommand {\C} {{\mathbb C}}

 \newcommand {\E} {{\mathcal E}}

\newcommand{\db}{\bar\partial}

\DeclareMathOperator*{\im}{im}

 \newtheorem{thm}[subsection]{Theorem}
 \newtheorem{cor}[subsection]{Corollary}
 \newtheorem{lemma}[subsection]{Lemma}

 \newtheorem{ex}[subsection]{Example}

\begin{document}

 \title{ Lefschetz decompositions for eigenforms on a K\"ahler manifold}

 \author{Donu Arapura}\thanks{Partially supported by the
   NSF} 
 \address{Department of Mathematics\\
   Purdue University\\
   West Lafayette, IN 47907\\
   U.S.A.}  

\maketitle

\begin{abstract}
  We show that the eigenspaces of  the  Laplacian $\Delta_k$ on $k$-forms on a
  compact K\"ahler manifold carry Hodge and Lefschetz
  decompositions. Among other consequences,  we show that the positive
  part of the spectrum of $\Delta_k$ lies in the spectrum of
  $\Delta_{k+1}$ for $k<\dim X$.
\end{abstract}

Given a compact Riemannian manifold $X$ without boundary,
let $\Delta_k$ denote the Laplacian on $k$-forms and $\lambda_1^{(k)}$ its smallest positive
eigenvalue.
We can ask how these numbers vary with $k$.
By differentiating eigenfunctions, we easily see that $\lambda_{1}^{(0)}\ge \lambda_{1}^{(1)}$.
For $k>1$, the situation is more complicated:  Takahashi \cite{tak2}  has shown that  the sign of
$\lambda_{1}^{(k)}-\lambda_{1}^{(0)}$ can be arbitrary for compact
Riemannian manifolds. More generally,  Guerini and  Savo \cite{g,gs}
have constructed examples, where the sequence
 $\lambda_1^{(2)},\lambda_1^{(3)},\ldots \lambda_1^{([\dim X/2])}$ can do just
about anything. The goal of this note is to show that  when
$X$ is   compact K\"ahler, the eigenspaces carry extra structure, and that
this imposes strong constraints on the eigenvalues
and their multiplicities. For instance, we show that the eigenvalues of
$\Delta_k$ occur
with even multiplicities when $k$ is odd. We also show that the positive part of the spectrum of $\Delta_k$
  is contained in the spectrum  of $\Delta_{k+1}$
 for all $k<\dim_\C X$. Therefore the sequence $\lambda_1^{(0)},\ldots,
 \lambda_1^{(\dim_\C X)}$ is  weakly decreasing. 

After this paper was submitted, it was brought to my attention that Jakobson, Strohmaier, 
 and Zelditch \cite{jsz} have also studied the spectra of K\"ahler
 manifolds, although for rather different reasons.

\section{Main theorems}

For the remainder of  this paper,
$X$ will denote a compact K\"ahler manifold of complex dimension $n$,
with   K\"ahler form $\omega$. Let $\Delta= d^*d+dd^*$ be the
Laplacian on  complex valued forms $\E^*$.
Standard arguments in  Hodge theory guarantee that  the spectrum of $\Delta$ is
discrete,
and the  eigenspaces
$$\E^*_\lambda= \{\alpha \in \E^*\mid \Delta\alpha=\lambda\alpha\}$$
are finite dimensional. Since $\Delta$ is positive and self adjoint, the
eigenvalues are nonnegative real.
We let $\E^k_\lambda$ and $\E^{(p,q)}_\lambda$ denote the intersection
of $\E^*_\lambda$ with the space of $k$ forms and $(p,q)$-forms
respectively. 

The proofs  of the following statements will naturally  hinge on the K\"ahler identities \cite{gh,
  wells}, which we recall below. 
We have
$$\Delta =
2(\partial\partial^*+\partial^*\partial)=2(\db\db^*+\db^*\db)$$
which implies that it commutes 
with the projections $\pi^{pq},\pi^k:\E^*\to \E^{(p,q)},\E^k$.
The Laplacian $\Delta$  also commutes  with the Lefschetz operator
$L(-)=\omega\wedge-$ and its adjoint $\Lambda$. 
An additional set of identities implies that $L,\Lambda$ and $H=\sum (n-k)\pi^k$ together determine an
action of the Lie algebra $sl_2(\C)$ on $\E^*$.

\begin{thm}\label{thm:1}
For each $\lambda$, there is a Hodge decomposition
\begin{equation}
  \label{eq:1}
 \E^k_\lambda= \bigoplus_{p+q=k}\E^{(p,q)}_\lambda 
\end{equation}
\begin{equation}
  \label{eq:2}
\overline{ \E_\lambda^{(p,q)}}
=\E_\lambda^{(q,p)}
\end{equation}
If $i>0$,  there is a hard Lefschetz isomorphism
\begin{equation}
  \label{eq:3}
\omega^i\wedge:\E^{n-i}_\lambda\stackrel{\sim}{\longrightarrow}\E^{n+i}_\lambda  
\end{equation}

\end{thm}

\begin{proof}
\eqref{eq:1} follows from the fact that
 $\Delta$ commutes  with $\pi^{pq}$. Since $\Delta$ and $\lambda$ are real, we obtain
\eqref{eq:2}. 
The proof of \eqref{eq:3}  is identical to the usual proof  of the hard
Lefschetz theorem \cite[pp 118-122]{gh}.
The key point is  that by representation theory, $L^i$ maps
$V\cap \E^{n-i}$ isomorphically to $V\cap \E^{n+i}$ for any
$sl_2(\C)$-submodule $V\subset \E^*$.
Applying this to the  subspace $V=\E_\lambda$, which  is an
$sl_2(\C)$-submodule because $L,\Lambda,H$ commute with $\Delta$,
proves (3).
\end{proof}

The first part of the theorem can be rephrased as saying that $\E_\lambda^k$ is a real Hodge structure of weight $k$.
We define the multiplicities $h^{pq}_\lambda= \dim \E^{(p,q)}_\lambda$
and $b^i_{\lambda}= \dim \E^i_\lambda$. When $\lambda=0$, these are
the usual Hodge and Betti numbers. In general, they depend on the metric.
These numbers share many properties of ordinary   Hodge
and Betti numbers:

\begin{cor}\label{cor:1}
For each $\lambda$,
  \begin{enumerate}
  \item[(a)] $b^k_{\lambda} =\sum_{p+q=k} h^{pq}_\lambda$
\item[(b)] $h^{pq}_\lambda= h^{qp}_\lambda$
\item[(c)] $b^k_{\lambda}$ is even if $k$ is odd.
\item[(d)] $b^{2n-k}_{\lambda}=b^k_{\lambda}$.
\item[(e)] if $k<n$,  $b^k_{\lambda}\le b^{k+2}_{\lambda}$
\item[(f)] if $p+q<n$, $h^{pq}_\lambda= h^{n-p,n-q}_\lambda$.
\item[(g)]  if $p+q<n$, $h^{pq}_\lambda\le h^{p+1,q+1}_\lambda$
  \end{enumerate}

\end{cor}

\begin{proof}
  The first four statements are immediate. For (e) we use the
  fact that $\omega\wedge\colon \E^k_\lambda\to \E^{k+2}_\lambda$ is an
  injection by \eqref{eq:3}.
 For (f) and (g), we use \eqref{eq:3}, and observe that $\omega^i\wedge -$ shifts the bigrading
  by $(i,i)$.
\end{proof}

The above results can be visualized in terms of the geometry of the ``Hodge diamond''.
When $\lambda>0$, there are some new patterns as well. We start with a
warm up.

\begin{lemma}
  If $\lambda$ is a positive eigenvalue of $\Delta_0$, then
  $h^{0,1}_\lambda\ge h^{0,0}_\lambda= b^0_\lambda$ and
  $b^1_\lambda\ge 2b^0_\lambda$.
\end{lemma}

\begin{proof}
The map $\db:\E_\lambda^0\to
  \E_\lambda^{0,1}$ is injective, because the kernel consists
  of  global holomorphic eigenfunctions which are necessarily
  constant and therefore $0$. This implies the first inequality, which in turn
  implies the second.
\end{proof}

We will give an extension to higher degrees, but first we start with a lemma.

\begin{lemma}\label{lemma:thm2}
  If $\lambda>0$,
$$\E_\lambda^k=d\E_\lambda^{k-1}\oplus d^*\E_\lambda^{k+1}$$
and
$$\E_\lambda^{(p,q)}=\db \E_\lambda^{(p,q-1)}\oplus
\db^*\E_\lambda^{(p,q+1)}$$
\end{lemma}

\begin{proof}
  By standard Hodge theory \cite{gh, wells},  we have the decompositions
$$\E^k = \E_0\oplus d\E^{k-1}\oplus d^*\E^{k+1}$$
$$\E^{k\pm 1}=\bigoplus_{i=1}^\infty
\E^{k\pm 1}_{\lambda_i^{(k\pm 1)}}$$
These can be combined to yield  the decomposition 
$$\E^k = \E^k_0\oplus \bigoplus_{i=1}^\infty
d\E^{k-1}_{\lambda_i^{(k-1)}}\oplus \bigoplus_{i=1}^\infty
d^*\E^{k+1}_{\lambda_i^{(k+1)}}$$
Since $d\E^{k-1}_\lambda,d^*\E^{k+1}_\lambda\subset \E^k_\lambda$, the
first part of the lemma
$$\E_\lambda^k=d\E_\lambda^{k-1}\oplus d^*\E_\lambda^{k+1}$$
follows immediately.
The proof of the second part is identical.
\end{proof}

\begin{thm}\label{thm:2}
 Suppose that $\lambda>0$.
\begin{enumerate}
\item[(a)] For all $k$, $b^k_\lambda\le b^{k-1}_\lambda+b^{k+1}_\lambda$

\item[(b)] If $p+q<n$, then $h^{pq}_\lambda\le h^{p+1,q}_\lambda+h^{p,q+1}_\lambda$.
\item[(c)]  If  $k<n$, then $b^k_{\lambda}\le b^{k+1}_{\lambda}$
\end{enumerate}

\end{thm}

\begin{proof}
  The first statement is an immediate consequence of lemma \ref{lemma:thm2}.

By lemma \ref{lemma:thm2}, we have a  direct sum 
$\E_\lambda^{(p,q)}= \E_{\im \db,\lambda}^{(p,q)}\oplus \E_{\im \db^*,\lambda}^{(p,q)}$ of the $\db$-exact
$\E_{\im\db,\lambda}^{(p,q)}:=\db \E_\lambda^{(p,q-1)} $ and $\db$-coexact
$\E_{\im \db^*,\lambda}^{(p,q)}:=\db^*\E_\lambda^{(p,q+1)} $ parts. We denote the dimensions of these spaces by
$h^{pq}_{\im \db,\lambda}$ and $h^{pq}_{\im \db^*,\lambda}$ respectively.

  Suppose that $\alpha\in \E_{\im \db^*,\lambda}^{(p,q)}$ then we can
  write $\alpha=\db^*\beta$.
We have
$\db\alpha\in
  \E^{(p,q+1)}_\lambda$ because $\db$ and $\Delta$ commute. Suppose that
  $\db\alpha=0$. Then 
   $$\alpha=\frac{1}{\lambda}\Delta\alpha =\frac{2}{\lambda}(\db^*\db+\db\db^*)\alpha
=\frac{2}{\lambda}\db\db^*\alpha =\frac{2}{\lambda}\db(\db^*)^2\beta$$
This is zero, because $\langle (\db^*)^2\beta,\xi\rangle =\langle
\beta,\db^2\xi\rangle=0$ for any $\xi$.
Thus the map 
$$\db: \E_{\im \db^*,\lambda}^{(p,q)}\hookrightarrow
\E_{\im \db,\lambda}^{(p,q+1)}$$
is injective. Although, we will not need it, it  is worth noting that
it also surjective because 
$$\E_{\im \db,\lambda}^{(p,q+1)}=\db(\E_{\im \db,\lambda}^{(p,q)}\oplus \E_{\im \db^*,\lambda}^{(p,q)})=\db \E_{\im \db^*,\lambda}^{(p,q)}$$ 
Therefore
\begin{equation}
  \label{eq:6}
 h_{\im \db^*,\lambda}^{p,q}= h_{\im \db,\lambda}^{p,q+1}
\end{equation}
for all $p,q$. 
We now assume that  $p+q<n$.  We will also establish an inequality 
\begin{equation}
  \label{eq:7}
   h_{\im \db,\lambda}^{p,q}\le  h_{\im \db^*\lambda}^{p+1,q}
\end{equation}
Let
$\alpha=\db\beta\in \E_{\im \db,\lambda}^{(p,q)}$ be a nonzero
element. The previous theorem shows that $\gamma=\omega\wedge\alpha$
is a nonzero element of $\E_\lambda$.
The form $\db^*\gamma\not=0$, since otherwise
$$\gamma=\frac{1}{\lambda}\Delta\gamma= \frac{2}{\lambda}\db^*\db^2(\omega\wedge \beta)=0$$
Thus we have proved that the map
$$
\E_{\im \db,\lambda}^{(p,q)}\to \E_{\im \db^*,\lambda}^{(p+1,q)}
$$
given by $\alpha\mapsto \db^*(\omega\wedge \alpha)$ is
injective. Equation \eqref{eq:7} is an immediate consequence.
Adding  \eqref{eq:6} and \eqref{eq:7} yields
\begin{equation}
  \label{eq:8}
  h^{pq}_\lambda \le  h_{\im \db,\lambda}^{p,q+1}+ h_{\im \db^*\lambda}^{p+1,q}
\end{equation}
which implies (b). Equation \eqref{eq:8} also implies
\begin{eqnarray*}
b^k_\lambda  &=& h^{0,k}_\lambda+h^{1,k-1}_\lambda+\ldots \\
 &\le& (h^{0,k+1}_{\im \db, \lambda} + h^{1,k}_{\im \db^*,\lambda})+(
 h^{1,k}_{\im \db, \lambda}+ h^{2,k-1}_{\im \db^*,\lambda} )+ \ldots\\
&=& h^{0,k+1}_{\im \db, \lambda} + (h^{1,k}_{\im \db^*,\lambda}+
 h^{1,k}_{\im \db, \lambda})+ \ldots\\
&\le& h^{0,k+1}_{\lambda} + h^{1,k}_{\lambda}+ \ldots\\
&= & b^{k+1}_\lambda
\end{eqnarray*}

\end{proof}

By the positive spectrum of an operator,
we mean the set of its positive eigenvalues (considered without multiplicity).

\begin{cor}
 The positive spectrum of $\Delta_{k}$ is
  contained in the union of the spectra of $\Delta_{k-1}$ and
  $\Delta_{k+1}$.
\end{cor}

Let $\lambda_1^{(k)}$ 
 denote the first strictly positive eigenvalue of
$\Delta_k=\Delta|_{\E^k}$.  

\begin{cor}
If $k<n$, the positive spectrum  of $\Delta_k$ is contained
in the positive spectrum of $\Delta_{k+1}$. Consequently,
$\lambda_1^{(0)}\ge \lambda_1^{(1)}\ge\ldots \ge\lambda_1^{(n)}$.
\end{cor}

We can show that the positive  spectra
of $\Delta_k$ coincide for certain values of $k$.

\begin{cor}
  The positive spectra of  $\Delta_{n-1}$, $\Delta_n$ and
  $\Delta_{n+1}$ coincide. In particular when $n=1$,  the positive
  spectra of all the Laplacians coincide. 
\end{cor}

\begin{proof}
If $\lambda>0$, then the inequalities
$$b_\lambda^{n-1}=b_\lambda^{n+1}$$
$$b_\lambda^{n-1}\le b_\lambda^{n}\le b_\lambda^{n-1}+b_\lambda^{n+1}=2b_\lambda^{n-1}$$
follow from theorems  \ref{thm:1} and \ref{thm:2}. These imply the corollary.

\end{proof}

The spectra are difficult to calculate in general, although there is
at least one case where it is straight forward.

\begin{ex}
Let  $L\subset \C^n$ be a lattice
with dual lattice $L^*$ with respect to the Euclidean inner product. The spectrum of each
  $\Delta_k$ on the flat torus $\C^n/L$ 
is easily calculated to be the same set $\{4\pi^2 ||v||^2\mid v\in L^*\}$,
cf. \cite[pp 146-148]{berger}.
\end{ex}

\end{document}